\documentclass[12pt,a4paper]{amsart}
\usepackage{mathrsfs}
\usepackage{amssymb,amsmath,amsthm,color}
\usepackage{graphicx,mcite}
\usepackage{hyperref}
\usepackage{url}
\usepackage{setspace}
\usepackage{enumerate}

\newtheorem{theorem}{Theorem}[section]
\newtheorem{lemma}[theorem]{Lemma}
\newtheorem{proof of lemma}[theorem]{Proof of Lemma}
\newtheorem{proposition}[theorem]{Proposition}

\theoremstyle{definition}

\usepackage[headheight=110pt,top=1.4in, bottom=1.4in, left=1.3in, right=1.3in]{geometry}

\newtheorem{remark}[theorem]{Remark}

\numberwithin{equation}{section}

\begin{document}

\title[Fourier nonuniqueness sets for the hyperbola]
{Fourier nonuniqueness sets for the hyperbola and the Perron-Frobenius operators}
%\runningtitle{Sets of Injectivity for Weighted Twisted Spherical Means}

%    Information for first author
\author{Deb Kumar Giri}
 %Address of record for the research reported here
%\address{Deb Kumar Giri, Department of Mathematics, Indian Institute of Technology, Guwahati, India 781039.}

\address{Deb Kumar Giri, Department of Mathematics, Indian Institute of Science, Bangalore-560012, India.}
%    Current address
%\curraddr{Department of Mathematics and Statistics,
%Case Western Reserve University, Cleveland, Ohio 43403}
\email{debkumarg@iisc.ac.in}
%\thanks will become a 1st page footnote.
%\thanks{The first author was supported in part by NSF Grant \#000000.}

%    Information for second author

%\address{Rama Rawat, Department of Mathematics, Indian Institute of Technology, Kanpur-208016, India.}
%    Current address
%\curraddr{Department of Mathematics and Statistics,
%Case Western Reserve University, Cleveland, Ohio 43403}
%\email{rrawat@iitk.ac.in}
%\thanks will become a 1st page footnote.
%\thanks{The first author was supported in part by NSF Grant \#000000.}

%    General info

\date{\today}


\begin{abstract}
Let $\Gamma$ be a smooth curve or finite disjoint union of smooth curves in the plane and $\Lambda$ 
be any subset of the plane. Let $\mathcal X(\Gamma)$ be the space of all finite complex-valued 
Borel measures in the plane which are supported on $\Gamma$ and are absolutely 
continuous with respect to the arc length measure on $\Gamma.$ Let   
$\mathcal{AC}(\Gamma,\Lambda)=\{\mu\in \mathcal{X}(\Gamma) : \hat\mu|_{\Lambda}=0\},$  
then we say that $\Lambda$ is a \textit{Fourier uniqueness set} for $\Gamma$ or $(\Gamma,\Lambda)$ 
is a Heisenberg uniqueness pair, if $\mathcal{AC}(\Gamma,\Lambda)=\{0\}.$ 
In particular, let $\Gamma$ be the hyperbola $\{(x,y)\in\mathbb R^2 : xy=1\}$ and 
$\Lambda_\beta$ be the lattice-cross in $\mathbb R^2$ is defined by 
$\Lambda_\beta=\left(\mathbb Z\times\{0\}\right)\cup\left(\{0\}\times\beta\mathbb Z\right),$  
where $\beta$ is a positive real. Then Canto-Mart\'in, Hedenmalm and Montes-Rodr\'iguez has shown 
that the space $\mathcal{AC}\left(\Gamma,\Lambda_\beta\right)$ is infinite-dimensional for 
$\beta>1.$ Further, they considered the branch 
$\Gamma_+=\{(x,y)\in\mathbb R^2 : ~xy=1,~x>0\}$ of the hyperbola $xy=1$ and the 
lattice-cross $\Lambda_{\gamma}=\left(2\mathbb Z\times\{0\}\right)\cup\left(\{0\}\times2\gamma\mathbb Z\right),$ where $\gamma$ is a positive real, and prove that $\mathcal{AC}(\Gamma_+,\Lambda_\gamma)$ 
is infinite-dimensional for $\gamma>1.$ 
\smallskip

\noindent In this paper, we prove the following results:  
\begin{enumerate}[(a)]
\item For a rational perturbation of $\Lambda_\beta$ namely,
$\Lambda_\beta^\theta=\left((\mathbb Z+\{\theta\})\times\{0\}\right)\cup\left(\{0\}\times\beta\mathbb Z\right),$ 
where $\theta=1/{p},~\text{for some}~{p}\in\mathbb N,$ and $\beta$ is a positive real,   
$\mathcal{AC}\left(\Gamma,\Lambda_\beta^\theta\right)$ is infinite-dimensional whenever $\beta>p.$
\smallskip

\item For a rational perturbation of $\Lambda_\gamma$ namely,
$\Lambda_\gamma^\theta=\left((2\mathbb Z+\{2\theta\})\times\{0\}\right)\cup\left(\{0\}
\times2\gamma\mathbb Z\right),$ where $\theta=1/q,~\text{for some}~q\in\mathbb N,$ 
and $\gamma$ is a positive real, $\mathcal{AC}\left(\Gamma_+,\Lambda_\gamma^\theta\right)$ 
is infinite-dimensional whenever $\gamma>q.$
\end{enumerate}
\end{abstract}

\maketitle

\section{Introduction}\label{section1}

\subsection{Heisenberg uniqueness pairs.} 
The uncertainty principle for Fourier transform states that a nonzero function and  
its Fourier transform both cannot be too concentrated at the same time (for the details 
see \cite{B, HJ, Hei}). The notion of Heisenberg uniqueness pair introduced recently by 
Hedenmalm and Montes-Rodr\'iguez as a version of this uncertainty principle. 
Further, the concept of Heisenberg uniqueness pair has significant similarity with mutually 
annihilating pairs of Borel measurable sets of positive measures. To describe this, consider a 
pair of Borel measurable sets $\mathcal S,\Sigma\subseteq\mathbb R.$ Then $(\mathcal S,\Sigma)$ 
forms a mutually annihilating pair if for any $\varphi\in L^2(\mathbb R)$ such that
$\text{supp}~\varphi\subset\mathcal S$ and whose Fourier transform $\hat\varphi$
supported on $\Sigma,$ implies $\varphi$ is identically zero (for more details see \cite{HJ}).

\smallskip
 
\noindent\textit{Heisenberg uniqueness pair.}  
In \cite{HR}, Hedenmalm and Montes-Rodr\'iguez proposes the following: Let $\Gamma$ be a smooth 
curve in $\mathbb R^2$ and $\Lambda$ be a subset of $\mathbb R^2.$ Let $\mathcal X(\Gamma)$ be the 
space of all finite complex-valued Borel measures $\mu$ in $\mathbb R^2$ which are supported on 
$\Gamma$ and are absolutely continuous with respect to the arc length measure on $\Gamma.$ For 
$(\xi,\eta)\in\mathbb R^2,$ the Fourier transform of $\mu$ is defined by
\begin{equation}\label{eq01}
\hat\mu{(\xi,\eta)}=\int_\Gamma e^{\pi i(x\xi+ y\eta)}d\mu(x,y).
\end{equation}
Let $\mathcal{AC}(\Gamma,\Lambda)=\{\mu\in \mathcal{X}(\Gamma) : \hat\mu|_{\Lambda}=0\},$ then 
following \cite{HR}, $\left(\Gamma, \Lambda\right)$ is said to be a \textit{Heisenberg uniqueness pair} 
(HUP) if $\mathcal{AC}(\Gamma,\Lambda)=\{0\}.$ In this case, since $\Lambda$ \textit{determine} 
the measures $\mu\in \mathcal{X}(\Gamma),$ we say that $\Lambda$ is a \textit{Fourier uniqueness set} 
for $\Gamma.$ The definition of Heisenberg uniqueness pair can be extended for more general 
measures but here we restrict our attention to those which are only absolutely continuous.  
Heisenberg uniqueness pairs satisfy the following invariance properties: 
\begin{enumerate}[(\textit{inv}-1)]
\item For any points $u_0, v_0\in\mathbb R^2,$ $\Big(\Gamma+ \{u_0\},\Lambda+\{v_0\}\Big)$ is a HUP if and only if  
$(\Gamma,\Lambda)$ is a HUP. 
\smallskip

\item $(\Gamma,\Lambda)$ is a HUP if and only if $\Big(T^{-1}(\Gamma),T^\ast(\Lambda)\Big)$  
is a HUP, where $T : \mathbb R^2\rightarrow \mathbb R^2$ be an invertible linear 
transform with adjoint $T^\ast.$ 
\end{enumerate}

\noindent{\bf\textit{The dual formulation.}}  $(\Gamma,\Lambda)$ is a HUP if and only if the 
subspace of all linear span of the functions $\{e^{\pi i(x\xi+y\eta)} :~(\xi,\eta)\in\Lambda\}$ 
is weak-star dense in $L^{\infty}(\Gamma).$ 

\smallskip

Many examples of Heisenberg uniqueness pair have been obtained in the plane as 
well as in the higher dimensional Euclidean spaces. The Heisenberg uniqueness 
pairs $(\Gamma,\Lambda)$ in which $\Gamma$ is the union of any two parallel lines 
in the plane was investigated in \cite{HR}. In \cite{Ba}, Babot studied the HUP 
in which $\Gamma$ is the certain system of three parallel lines. Later, the author 
has studied the HUP corresponding to a certain system of four parallel lines together 
with some algebraic curves (see \cite{GR1}). Further, the cases in which $\Gamma$ is the 
union of a certain system of finitely many parallel lines were studied in \cite{Bag, GR2}.
The cases in which $\Gamma$ is the unit circle were independently investigated in \cite{L, S1}. 
Later, Gonzal\'ez Vieli (see \cite{Gon}) generalized the cases for the circle to the higher
dimension using the properties of the Bessel functions $J_{(n+2k-2)/2};~k\in\mathbb {Z_+}.$ 
In \cite{Sri}, Srivastava studied the cases for the sphere in which $\Lambda$ is the cone as 
well as the cone does not contain in the zero sets of any homogeneous harmonic polynomial 
on $\mathbb R^n.$ In \cite{S2}, Sj\"{o}lin investigated the HUP corresponding to the 
parabola while the author (see \cite{CGGS}) studied certain exponential surfaces and connect 
the notion of HUP to the Euclidean motion groups. The dynamical system approach was used 
in \cite{JK} to study the cases for hyperbola, polygon, ellipse, and graph of $\varphi(t)=|t|^\alpha,$
whenever $\alpha>0.$ Later, Gr\"{o}chenig and Jaming \cite{GJ} solved the cases corresponding 
to the quadratic surface. In a recent article \cite{GS1}, the authors had extended the notion 
of Heisenberg uniqueness pair to the Heisenberg group. 
As a major development along this direction, in \cite{MHR, HR, HR2, HR3}, 
the dynamics of Gauss-type maps, Ergodic theory, Klein-Gordon equation, and the Perron-Frobenius 
operators were used to studying HUP which advanced the theory. 
\smallskip

In this paper, the problems (Theorem \ref{th6} and Theorem \ref{th11}) we consider is inspired by the articles \cite{MHR, GRA, HR}, and closely follows the methods of \cite{MHR} with moderate modifications at appropriate places. We will follow definitions and notation from \cite{MHR} as much as possible.

\subsection{Nonuniqueness sets for the hyperbola}

Let $\Gamma=\{(x,y)\in\mathbb R^2 : ~xy=1\}$ be the hyperbola and $\mu\in \mathcal{X}(\Gamma),$ 
then there exists $g\in L^1(\mathbb R, \sqrt{1+1/t^4}~dt)$ such that for bounded and 
continuous function $\varphi$ on $\mathbb R^2,$
\[\int_\Gamma\varphi(x,y)d\mu(x,y)=\int_{\mathbb R\setminus\{0\}}\varphi(t,1/t)g(t)\sqrt{1+1/t^4}dt.\] 
In particular, for $(\xi,\eta)\in\mathbb R^2,$ the Fourier transform of $\mu$ can be expressed as 
\[\hat\mu{(\xi,\eta)}=\int_{\mathbb R\setminus\{0\}} e^{\pi i(\xi t+ \eta/t)}f(t)dt,~~ 
~\text{where}~f(t):=g(t)\sqrt{1+1/t^4}~\in L^1(\mathbb R).\]
\smallskip

As a first known result on HUP, in \cite{HR}, Hedenmalm and Montes-Rodr\'iguez have studied that 
some lattice-cross in the plane is a \textit{Fourier uniqueness set} for the hyperbola and have 
proved the following result. 
\begin{theorem}\cite{HR}\label{th1}
Let $\Gamma=\{(x,y)\in\mathbb R^2 : xy=1\}$ be the hyperbola and $\Lambda_\beta$ be the 
lattice-cross 
$\Lambda_{\beta}:=\left(\mathbb Z\times\{0\}\right)\cup\left(\{0\}\times\beta\mathbb Z\right),$
where $\beta$ is a positive real. Then $\mathcal{AC}\left(\Gamma,\Lambda_\beta\right)=\{0\}$ 
if and only if $\beta\leq1$.
\end{theorem}

In \cite{MHR}, Canto-Mart\'in, Hedenmalm and Montes-Rodr\'iguez have studied that some lattice-cross 
in $\mathbb R^2,$ is a \textit{Fourier nonuniqueness set} for the hyperbola and have proved the following 
result.

\begin{theorem}\cite{MHR}\label{th2}
Let $\Gamma$ be the hyperbola $xy=1$ and $\Lambda_\beta$ be the lattice-cross 
$\Lambda_{\beta}:=\left(\mathbb Z\times\{0\}\right)\cup\left(\{0\}\times\beta\mathbb Z\right),$ 
where $\beta$ is a positive real. Then $\mathcal{AC}(\Gamma,\Lambda_\beta)$ 
is infinite-dimensional for $\beta>1.$ 
\end{theorem}

Let $\mathcal M_\beta$ be the subspace of all linear span of the functions 
$\{e_n(x):=e^{\pi inx};~n\in\mathbb Z\}\cup\{e^{\beta}_n(x):
=e^{\pi in\beta/x};~n\in\mathbb Z\}$ in $L^\infty(\mathbb R),$  
where $\beta$ is a positive real. The codimension of the weak-star closure 
of $\mathcal M_\beta$ in $L^\infty(\mathbb R)$ is the dimension of its 
pre-annihilator space  
\[\mathcal M_\beta^\perp:=\left\{f\in L^1(\mathbb R)~:~\int_\mathbb R f(x)e_n(x)dx=
\int_\mathbb R f(x)e^{\beta}_n(x)dx=0~\text{for~all}~n\in\mathbb Z\right\}.\]
By dual formulation, Theorem \ref{th1} is equivalent to the following 
density result. 
\begin{theorem}\cite{HR}\label{th3}
The space $\mathcal M_\beta$ is weak-star dense in $L^\infty(\mathbb R)$ 
if and only if $0<\beta\leq{1}.$
\end{theorem}
Similarly, by dual formulation, Theorem \ref{th2} is equivalent to the following 
density result.
\begin{theorem}\cite{MHR}\label{th4}
$\mathcal M_\beta^\perp$ is an infinite-dimensional subspace of $L^1(\mathbb R)$ 
for $1<\beta<\infty.$ 
\end{theorem} 

\noindent Next, we consider a rational perturbation of the lattice-cross $\Lambda_\beta,$ 
namely that   
\begin{equation}\label{eq02}
\Lambda_\beta^\theta:=\left((\mathbb Z+\{\theta\})\times\{0\}\right)\cup\left(\{0\}\times\beta\mathbb Z\right),
\end{equation}
where $\theta=1/{p},~\text{for some}~{p}\in\mathbb N,$ and $\beta$ is a positive real. 
Let $\Gamma$ be the hyperbola $xy=1,$ then the following result shows that $\left(\Gamma,\Lambda_\beta^\theta\right)$ is a Heisenberg uniqueness pair for $0<\beta\leq{p}.$ In other words, $\Lambda :=\Lambda_\beta^\theta$ 
is a \textit{Fourier uniqueness set} for the hyperbola $\Gamma$ whenever $\beta\leq p.$ 

\begin{theorem}\cite{GRA}\label{th5}
$\mathcal{AC}\left(\Gamma,\Lambda_\beta^\theta\right)=\{0\}$ if and only if $0<\beta\leq{p}.$
\end{theorem}

\begin{remark}
\begin{enumerate}[(a)]
\item It is rather surprising that the condition on $\beta$ depends on $\theta.$ The proof of 
Theorem \ref{th5} works along the same lines as in \cite{HR} but with moderate modifications 
at appropriate places. The question is remains open for irrational values of $\theta.$   
\smallskip

\item The notion of Heisenberg uniqueness pair may be extended for more general finite complex-valued 
Borel measures $\mu$ in $\mathbb R^2$ which are supported on $\Gamma$ without assuming absolute continuity 
with respect to the arc length measure on $\Gamma.$ But for Theorem \ref{th5}, the measures $\mu$ must be 
absolutely continuous with respect to the arc length measure on the hyperbola, without this assumption, 
Theorem \ref{th5} is not true.
\end{enumerate}
\end{remark}

Next, we state a result of this paper which is a variant of Theorem \ref{th2}.

\begin{theorem}\label{th6}
Let $\Gamma$ be the hyperbola $xy=1$ and $\Lambda_\beta^\theta$ be the lattice-cross 
defined in (\ref{eq02}). Then $\mathcal{AC}\left(\Gamma,\Lambda_\beta^\theta\right)$ is 
infinite-dimensional for $\beta>p.$ 
\end{theorem} 

\begin{remark}
\begin{enumerate}[(a)]
\item Theorem \ref{th6} asserts that $\Lambda :=\Lambda_\beta^\theta$ is a \textit{Fourier nonuniqueness 
set} for the hyperbola $xy=1$ whenever $\beta>p.$ The presence of $\theta$ showing up in the condition 
of $\beta$ which is somewhat unexpected. The proof of Theorem \ref{th6} works along the same lines as 
in \cite{MHR} but with modifications at appropriate places.
\smallskip

\item As a corollary to Theorem \ref{th6}, let $\Gamma$ be the hyperbola $xy=1$ and $\Lambda_\beta^\zeta$ 
be the set $\Lambda_\beta^\zeta:=\left((\mathbb Z+\{\zeta\})\times\{0\}\right)\cup\left(\{0\}\times\beta\mathbb Z\right),$ where $\zeta=r/p,~\text{for some}~p\in\mathbb N~\text{and}~r\in\mathbb Z$ with 
${\text{gcd(p,r)=1}}$ and $\beta$ is a positive real. Then $\mathcal{AC}\left(\Gamma,\Lambda_\beta^\zeta\right)$ 
is infinite-dimensional for $\beta>p.$
\smallskip

\item Let $\Gamma$ be the hyperbola $xy=1,$ then any $\mu\in\mathcal{AC}\left(\Gamma,\Lambda_\beta^\theta\right),$ $u:=\hat\mu$ is a solution of the one-dimensional Klein-Gordon equation: $\left(\partial_\xi\partial_\eta+\pi^2\right)u(\xi,\eta)=0$ in the sense of distributions. Theorem \ref{th6} says that for $\beta>p,$ 
the solution space of the above partial differential equation is infinite-dimensional.
\end{enumerate}
\end{remark} 
\smallskip

Let $\mathcal F_\beta$ be the subspace of all linear span of the functions $e^{p}_n(x),e^{\beta}_n(x);~n\in\mathbb Z$ in $L^\infty(\mathbb R),$ where $e^{p}_n(x):=e^{\pi i(n+1/p)x}$ and $e^{\beta}_n(x):=e^{\pi in\beta/x}$ 
with $p\in\mathbb N$ and $\beta$ is a positive real. The codimension of the weak-star closure of 
$\mathcal F_\beta$ in $L^\infty(\mathbb R)$ is the dimension of its pre-annihilator space 
$\mathcal F_\beta^\perp.$ By dual formulation, Theorem \ref{th5} is equivalent to the following result.

\begin{theorem}\cite{GRA}\label{th7}
The space $\mathcal F_\beta$ is weak-star dense in $L^\infty(\mathbb R)$ 
if and only if $0<\beta\leq{p}.$
\end{theorem}
Similarly, Theorem \ref{th6} is equivalent to the following density result.
\begin{theorem}\label{th8}
$\mathcal F_\beta^\perp$ is an infinite-dimensional subspace of 
$L^1(\mathbb R)$ for $p<\beta<\infty.$ 
\end{theorem}

\subsection{Nonuniqueness sets for the branch of the hyperbola.}  
Let $\Gamma_+=\{(x,y)\in\mathbb R^2 : ~xy=1,~x>0\}$ be the branch of the hyperbola 
and $\mu\in \mathcal{X}(\Gamma_+),$ then there exists $g\in L^1(\mathbb R_+, \sqrt{1+1/t^4}~dt)$ 
such that for bounded and continuous function $\varphi$ on $\mathbb R^2,$
\[\int_{\Gamma_+}\varphi(x,y)d\mu(x,y)=\int_{\mathbb R_+\setminus\{0\}}\varphi(t,1/t)g(t)\sqrt{1+1/t^4}dt.\] 
In particular, for $(\xi,\eta)\in\mathbb R^2,$ the Fourier transform of $\mu$ can be expressed as 
\[\hat\mu{(\xi,\eta)}=\int_{\mathbb R_+\setminus\{0\}} e^{\pi i(\xi t+ \eta/t)}f(t)dt,~~ 
\text{where}~f(t):=g(t)\sqrt{1+1/t^4}~\in L^1(\mathbb R_+).\]   
In \cite{MHR}, Canto-Mart\'in, Hedenmalm and Montes-Rodr\'iguez have studied that some lattice-cross 
in $\mathbb R^2,$ is a \textit{Fourier nonuniqueness set} for $\Gamma_+$  
and have proved the following result.

\begin{theorem}\cite{MHR}\label{th9}
Let $\Gamma_+=\{(x,y)\in\mathbb R^2 : ~xy=1,~x>0\}$ be the branch of the hyperbola and $\Lambda_\gamma$ be the lattice-cross $\Lambda_{\gamma}:=\left(2\mathbb Z\times\{0\}\right)\cup\left(\{0\}\times2\gamma\mathbb Z\right),$ 
where $\gamma$ is a positive real. Then $\mathcal{AC}(\Gamma_+,\Lambda_\gamma)$ is infinite-dimensional for 
$\gamma>1.$ 
\end{theorem}

By dual formulation, Theorem \ref{th9} is equivalent to the following result.

\begin{theorem}\cite{MHR}\label{th10}
Let $\mathcal N_\gamma$ be the subspace of all linear span of the functions 
$\{e_n(x):=e^{2\pi inx};~n\in\mathbb Z\}\cup\{e^{\gamma}_n(x):
=e^{2\pi in\gamma/x};~n\in\mathbb Z\}$ in $L^\infty(\mathbb R_+),$  
where $\gamma$ is a positive real. Then the pre-annihilator space $\mathcal N_\gamma^\perp$ 
is infinite-dimensional for $\gamma>1.$
\end{theorem}

Next, we state a result of this paper which is a variant of Theorem \ref{th9}. 

\begin{theorem}\label{th11}
Let $\Gamma_+=\{(x,y)\in\mathbb R^2 : xy=1,~x>0\}$ be the branch of the hyperbola 
and $\Lambda_\gamma^\theta$ be the lattice-cross 
$\Lambda_\gamma^\theta:=\left((2\mathbb Z+\{2\theta\})\times\{0\}\right)\cup\left(\{0\}\times2\gamma\mathbb Z\right),$ where $\theta=1/q,~q\in\mathbb N,$ and $\gamma$ is a positive real.   
Then $\mathcal{AC}(\Gamma_+,\Lambda_\gamma^\theta)$ is infinite-dimensional for $\gamma>q.$ 
\end{theorem}

\begin{remark}
\begin{enumerate}[(a)]
\item Theorem \ref{th11} asserts that $\Lambda :=\Lambda_\gamma^\theta$ is a \textit{Fourier 
nonuniqueness set} for the branch $\Gamma_+$ whenever $\gamma>q.$ The presence of $\theta$ 
showing up in the condition of $\gamma$ which is somewhat unexpected. The proof of Theorem 
\ref{th3} works along the same lines as in \cite{MHR} but with modifications at appropriate 
places. The question is still open when $\theta$ is irrational. 
\smallskip

\item Let $\Lambda_\gamma^\theta$ be the lattice-cross $\left((2\mathbb Z+\{2\theta\})\times\{0\}\right)\cup\left(\{0\}\times2\gamma\mathbb Z\right),$ where $\theta=1/q,~q\in\mathbb N,$ 
and $\gamma$ is a positive real. It seems likely, that $\mathcal{AC}\left(\Gamma_+,\Lambda_\gamma^\theta\right)=\{0\}$ if and only if $\gamma< q,$ and for the critical case $\gamma=q,$  
$\mathcal{AC}\left(\Gamma_+,\Lambda_\gamma^\theta\right)$ is one-dimensional in analogy with 
the results in \cite{HR2}. The question is still open.
\end{enumerate}
\end{remark}
By duality, Theorem \ref{th11} is equivalent to the following density result.

\begin{theorem}\label{th12}
Let $\mathcal K_\gamma$ be the subspace of all linear span of the functions 
$\{e^{q}_n(x):=e^{2\pi i(n+1/q)x};~n\in\mathbb Z\}\cup\{e^{\gamma}_n(x):
=e^{2\pi in\gamma/x};~n\in\mathbb Z\}$ in $L^\infty(\mathbb R_+),$ where $q\in\mathbb N$  
and $\gamma$ is a positive real. Then the pre-annihilator space $\mathcal K_\gamma^\perp$ 
is infinite-dimensional for $\gamma>q.$
\end{theorem}

\subsection{The Perron-Frobenius operators}

In this section, we recall the definitions and notation related to the Perron-Frobenius operators 
associated with a $\mathcal{C}^2$-smooth piecewise monotonic transform from (\cite{MHR}, Section 3) 
as far as possible. The spectral property of the Perron-Frobenius operators has played a significant 
role in the HUP.

\subsubsection{} \textit{Perron-Frobenius operators on bounded intervals.}
\smallskip

\noindent Let $I\subset\mathbb R$ be a closed and bounded interval and $m$ be the Lebesgue measure 
defined on the $\sigma$-algebra of $I.$ Following (\cite{MHR}, Definition 3.2), a measurable map 
$\tau : I\rightarrow I$ is said to be a \textit{"partially filling $\mathcal{C}^2$-smooth piecewise 
monotonic transform"} if there exists a countable collection of pairwise disjoint open intervals 
say, $\{I_u\}_{u\in\mathcal U},$ where $\mathcal U$ is the index set, such that the following holds:  
\begin{enumerate}[(i)]
\item $m\left(I\setminus\bigcup\{I_u: u\in\mathcal U\}\right)=0,$

\item for any $u\in\mathcal U,$ the map $\tau_u:=\tau|_{I_u}$ is strictly monotone 
and can be extended to a $\mathcal{C}^2$-smooth function on $\bar{I_u}$ with $\tau_u'\neq0$ on $I_u^o,$

\item there exists a positive number say, $\delta$ such that $m\left(\tau(I_u)\right)\geq\delta$ 
for all $u\in\mathcal U.$ 
\end{enumerate}
\smallskip

Following (\cite{MHR}, Definition 3.1), $\tau$ is said to be a \textit{"filling $\mathcal{C}^2$-smooth 
piecewise monotonic transform"}, if the above conditions (i),(ii) holds for $\tau$ along with 
\smallskip

\noindent $(iii)'$ for every $u\in\mathcal U,$ the map $\tau_u : \bar{I_u}\rightarrow I$ is onto. 
\smallskip

Observe that, condition $(iii)'$ is much stronger than condition $(iii).$ In the above context, each 
$I_u$ is called a \textit{"fundamental interval"} and $\tau_u$ is the corresponding \textit{"branch"}.
\smallskip

The Koopman operator $\mathcal C_\tau : L^\infty(I)\rightarrow L^\infty(I)$ corresponding to a measurable 
map $\tau : I\rightarrow I$ is defined by letting \[\mathcal C_\tau[\varphi]=\varphi\circ\tau.\]
The Perron-Frobenius operator $\mathcal P_\tau : L^1(I)\rightarrow L^1(I)$ is the pre-dual adjoint 
(Banach space dual) of $\mathcal C_\tau$ is given by 
\begin{equation}\label{eq03}
\langle\mathcal P_\tau[\psi],\varphi\rangle_I=\langle\psi,\mathcal C_\tau[\varphi]\rangle_I,~~
\text{where}~\psi\in L^1(I)~\text{and}~\varphi\in L^\infty(I).
\end{equation}
The operator $\mathcal P_\tau$ is linear and a norm contraction on $L^1(I),$ therefore, its spectrum 
$\sigma(\mathcal P_\tau)$ is contained in the closed unit disk $\bar{\mathbb D}=\{\lambda\in\mathbb C:|\lambda|\leq1\}.$

\subsubsection{} \textit{The spectral decomposition of Perron-Frobenius operators.} 
\smallskip

\noindent\textit{Functions of bounded variation in one variable.} Let $I\subset\mathbb R$ be a closed 
bounded interval. For any function $h : I\rightarrow\mathbb C,$ the pointwise variation of 
$h$ in $I$ is defined by letting 
\[pV(h,I) :=\sup\limits_{\substack{t_1,\ldots,t_n\in I,\\t_1<\cdots<t_n;~n\geq2}}\left\{\sum\limits_{i=1}^{n-1}|h(t_{i+1})-h(t_{i})|\right\}.\] 
Then $h$ is said to be a function of bounded variation provided $pV(h,I)<\infty.$ Let 
$\text{BV}(I)$ denote the subspace of all functions in $L^1(I)$ such that the pointwise 
variation is finite. The essential variation $eV(h,I)$ of $h$ in which the pointwise 
variation is minimized in the equivalence class is defined by 
\[eV(h,I) :=\inf\left\{pV(\tilde h,I) : \tilde h=h~\text{except~ on ~a~ set~ of~Lebesgue~ 
measure~ zero}\right\}.\] 
For any $h\in\text{BV}(I),$ from (\cite{AFP}, Theorem 3.27) we get that the infimum in 
the expression of $eV(h,I)$ is achieved. Hence the space $\text{BV}(I)$ equipped with the norm 
\[\|h\|_\text{BV}=\|h\|_{L^1(I)}+eV(h,I),~~h\in\text{BV}(I),\] 
becomes a Banach space. 
\smallskip

Next, we state the spectral decomposition of Perron-Frobenius operators $\mathcal P_\tau$
associated to $\tau$ on $I.$ Let $\partial\bar{\mathbb D}$ denote the boundary of the closed unit disk 
$\bar{\mathbb D}$ and the point spectrum of $\mathcal P_\tau$ is denoted by $\sigma_{point}(\mathcal P_\tau).$ 
Then the following spectral decomposition for $\mathcal P_\tau$ is stated in (\cite{MHR}, Theorem C) which 
is a consequence of the Ionescu-Tulcea and Marinescu theorem (for details see \cite{BG,IG}). Recall the definition 
of $\mathcal U^m_\sharp;m\geq1$ from (\cite{MHR}, p. 39). In particular, for $m=1,$ we have $\mathcal U^m_\sharp
=\mathcal U.$ 
\smallskip

\noindent{\textbf{Theorem A.}} Let $\tau : I \rightarrow I$ is a \textit{partially filling 
$\mathcal C^2$-smooth piecewise monotonic transform} such that  
\begin{enumerate}[(i)]
\item \text{[uniform expansiveness]} there exists an integer $m\geq1$ and $\epsilon>0$ such 
that $|(\tau^m)'(x)|\geq 1+\epsilon$ for all $x\in\bigcup\{I_u : u\in\mathcal U^m_\sharp\},$
\smallskip

\item \text{[second derivative condition]} there exists $M>0$ such that 
$|\tau''(x)|\leq M |\tau'(x)|^2$ for all $x\in\bigcup\{I_u : u\in\mathcal U\},$
\end{enumerate}
then $\Lambda_\tau~:=\sigma_{point}(\mathcal P_\tau)\cap\partial\bar{\mathbb D}$ is a finite set namely,  
$\Lambda_\tau=\{\alpha_1,\ldots,\alpha_s\}$ and one of the eigenvalues is 1, say $\alpha_1=1.$  
Let $E_i$ denotes the eigenspace of $\mathcal P_\tau$ for the eigenvalue $\alpha_i,$ then $E_i$ 
is finite-dimensional and $E_i$ is contained in $\text{BV}(I).$ In addition, 
\[\mathcal P_\tau^n[h]=\sum\limits_{i=1}^s\alpha_i^n\mathcal P_{\tau,i}[h] + \mathcal Z_\tau^n[h],
~~h\in L^1(I),~~n=1,2,\ldots,\] 
where the operators $\mathcal P_{\tau,i}$ are projections onto $E_i,$ 
and the operator $\mathcal Z_\tau$ acts boundedly on $L^1(I)$ as well as on $\text{BV}(I).$ 
Moreover, $\mathcal Z_\tau$ acting on $\text{BV}(I)$ has spectral radius $<1.$

\section{Proof of Theorem \ref{th8}}

\subsection{Dynamics of a Gauss-type map}

\subsubsection{} \textit{A Gauss-type map.} For $t\in\mathbb R,$ the expression $\{t\}_2$ 
represent the unique number in $(-1,1]$ such that $t-\{t\}_2\in2\mathbb Z.$ We consider a 
Gauss-type map $U$ on the interval $(-p,p];~~p\in\mathbb N$ which is defined 
by letting 
\[U(x):=\begin{cases} 
      p\left\{-\frac{p}{x}\right\}_2,~x\neq0, \\
      \quad 0 \qquad\text{for}~ x=0. 
   \end{cases}
\] 
For $u\in\mathbb Z^\ast=\mathbb Z\setminus\{0\},$ the map $U$ can be explicitly written as  
$U(x)=p\left(2u-p/x\right)$ whenever $\frac{{p}}{2u+1}<x\leq\frac{p}{2u-1},$ 
and hence $U: \left(\frac{{p}}{2u+1},\frac{{p}}{2u-1}\right]\rightarrow(-p,p]$ 
is one-to-one and for $x\in(-p,p]\setminus\frac{p}{2\mathbb Z+1},$ the derivative of $U$ 
is $U'(x)=\frac{p^2}{x^2}.$ For a continuous $2p$-periodic function $\varphi$ on $\mathbb R$ 
and a finite complex-valued Borel measure $\nu$  on $(-p,p],$ the integral 
$\int_{(-p,p]}\varphi(x)d\nu(x)$ is well-defined. The above integral makes sense for all 
\textit{pseudo-continuous} functions on $(-p,p].$ 
 
\smallskip

Note that for a \textit{pseudo-continuous} 
function $\varphi$ on $(-p,p],$ $\varphi\circ U$ is \textit{pseudo-continuous}. Given $\lambda\in\mathbb C,$ 
a finite complex Borel measure $\nu$ on $(-p,p]$ is $(U,\lambda)$-invariant provided that
\[\int_{(-p,p]}\varphi\left(U(x)\right)d\nu(x)=\lambda\int_{(-p,p]}\varphi(x)d\nu(x)\]
holds for all \textit{pseudo-continuous} functions $\varphi,$ that is, 
$\lambda\nu=\nu(\{0\})\delta_0+\sum\limits_{u\in\mathbb Z^\ast}\nu_u,$ where $\delta_t$ denote  
the point mass at $t,$ and $d\nu_u(x)=d\nu\left(\frac{p^2}{2pu-x}\right).$ It is easy to see that, 
for $|\lambda|>1,$ there are no $(U,\lambda)$-invariant measures except zero measure.
\smallskip

In this work, we mainly study the properties the following map which is associated to the parameter $\beta.$  
For $0<\beta<\infty,$ the Gauss-type map $U_\beta : (-p,p]\rightarrow(-p,p]$ is defined by letting 
\begin{equation}
U_\beta(x):=\begin{cases} 
      p\left\{-\frac{\beta}{x}\right\}_2,~x\neq0, \\
      \quad 0 \qquad\text{for}~ x=0. 
   \end{cases}
\end{equation}
For $u\in\mathbb Z^\ast=\mathbb Z\setminus\{0\},$ on the interval $\left(\frac{\beta}{2u+1},\frac{\beta}{2u-1}\right],$ the map $U_\beta$ can be expressed as $U_\beta(x)=p\left(2u-\beta/x\right).$ In particular, $U_p=U.$ 
For $\lambda\in\partial\bar{\mathbb D},$   
a finite complex Borel measure $\nu$ on $(-p,p]$ is $(U_\beta,\lambda)$-invariant provided 
that $\lambda\nu=\nu(\{0\})\delta_0+\sum\limits_{u\in\mathbb Z^\ast}\nu_u,$ where 
$d\nu_u(x)=d\nu\left(\frac{p\beta}{2pu-x}\right).$

\subsubsection{} \textit{Dynamical properties of $U_\beta$ for $\beta>p.$}
\smallskip 

Denote $\beta_0:=\beta/p>1.$ In this section, we observe that $U_\beta$ is a \textit{partially 
filling $\mathcal C^2$-smooth piecewise monotonic transform} for $\beta_0>1.$ We need to find a 
unique absolutely continuous invariant probability measure for $U_\beta$ having positive density.
\smallskip
 
Let $\mathcal U:=\mathcal U_{\beta_0}$ denote the index set which contain the points $u\in\mathbb Z^*$ 
so that the associate fundamental interval is nonempty, that is, 
\[I_u :=\left(\frac{\beta}{2u+1},~\frac{\beta}{2u-1}\right)\cap(-p,~p)\neq \emptyset.\]   
If $\beta_0$ is an odd integer, then $\{\beta_0\}_2=1.$ In this case, it is easy to see that 
for all $u\in\mathcal U,$ where $\mathcal U$ be the set of all nonzero integers with 
$|u|\geq\frac{1}{2}(\beta_0+1),$ the fundamental intervals are given by 
\begin{equation}\label{eq198}
I_u :=\left(\frac{\beta}{2u+1},~\frac{\beta}{2u-1}\right)
\end{equation} 
so that $U_\beta(I_u)=(-p,~p)$ for all $u\in\mathcal U.$ Thus in this particular case,  
$U_\beta$ fulfill the \text{"filing"} condition for all $u\in\mathcal U.$ 
\smallskip
  
If $\beta_0$ is not an odd integer, then $-1<\{\beta_0\}_2<1.$ Denote $u_0 :=\frac{1}{2}(\beta_0-\{\beta_0\}_2),$ then $u_0\in\mathbb Z$ with $u_0\geq1.$ A simple calculation shows that for all 
$u\in\mathcal U\setminus\{\pm u_0\},$ where $\mathcal U$ be the set of all nonzero integers 
with $|u|\geq u_0,$ the fundamental intervals are given by 
\begin{equation}\label{eq199}
I_u :=\left(\frac{\beta}{2u+1},~\frac{\beta}{2u-1}\right)
\end{equation} 
so that $U_\beta(I_u)=(-p,~p)$ for all $u\in\mathcal U\setminus\{\pm u_0\}.$ Thus  
in this case, $U_\beta$ fulfill the \text{"filing"} condition for all $u\in\mathcal U$ 
except two branches corresponding to $\{\pm u_0\}.$ The edge fundamental intervals corresponding 
to $\pm u_0$ are explicitly given by 
\begin{equation}\label{eq200}
I_{u_0} :=\left(\frac{\beta}{2{u_0}+1},~p\right),~~I_{-u_0} :=\left(-p,~-\frac{\beta}{2{u_0}+1}\right).
\end{equation} 
In view of the above facts, we conclude that $U_\beta$ is a \textit{partially filling $\mathcal C^2$-smooth piecewise monotonic transform} for $\beta_0>1.$
\smallskip

The map $U_\beta$ satisfy the \textit{"uniform expansiveness"} condition with $m=1,$ because  
of the derivative $|U_\beta'(x)|=\frac{p\beta}{x^2}\geq \beta_0>1$ for all $x\in\{I_u : u\in\mathcal U\}.$   
Also, we have \textit{"second derivative condition"} due to $|U_\beta''(x)|\leq\frac{2}{p}|U_\beta'(x)|^2$ 
for all $x\in\{I_u : u\in\mathcal U\}.$ Our aim is to find a unique absolutely continuous invariant 
measure for $U_\beta$ which has a positive density. Since $U_\beta$ is a \textit{partially 
filling $\mathcal C^2$-smooth piecewise monotonic transform}, (see \cite{MHR}, p. 45, Remark 5.1(b)), 
there exists a $U_\beta$-invariant absolutely continuous probability measure, but it may not be unique. 
Although, form (\cite{MHR}, p. 45, Remark 5.2) we can get the uniqueness for \textit{filling $\mathcal C^2$
-smooth piecewise monotonic transform}. Also, $U_\beta$ may not always be a Markov map, therefore, 
to get the uniqueness of the absolutely continuous $U_\beta$-invariant measure, we have to prove 
the condition $(iii)$ of Adler's theorem which is stated in (\cite{MHR}, p. 46, Theorem D), for the details 
see \cite{BG, LY}.

\subsubsection{} \textit{The iterates of an interval.} 
\smallskip

The following result will help to get the unique ergodic $U_\beta$-invariant absolutely continuous 
probability measure with positive density.   

\begin{lemma}\label{lemma10}
$(p<\beta<\infty)$ Let $J_0\subset[-p,p]$ be any nonempty open interval, then for sufficiently large 
positive integers, namely that $n\geq n_0,$ we have ${\mathscr C_p}:=(-p,p)\subset U_\beta^n(J_0).$
\end{lemma}

\begin{proof}
We want to show that $U_\beta^n(J_0)$ will cover the open interval ${\mathscr C_p}$ for sufficiently 
large positive integers namely, $n\geq n_0.$ The proof will be carried out in the following cases. 
\smallskip

\noindent{\bf\textit{1. The case $\beta_0$ is an odd integer.}}
\smallskip

\noindent In this case $\beta_0\geq3$ and hence $\beta_0/2>1.$ The fundamental intervals $I_u$ are 
given by (\ref{eq198}) for all $u\in\mathcal U,$ where $\mathcal U$ be the set of all nonzero integers 
with $|u|\geq\frac{1}{2}(\beta_0+1).$ Here, we want to show that ${\mathscr C_p}\subset U_\beta^n(J_0)$ for 
sufficiently large $n.$ There are two possibilities: 
\smallskip

\noindent {\bf{(i).}} Suppose $J_0$ contains one of the fundamental intervals say, $I_u$
for some $u\in\mathcal U,$ then it follows that ${\mathscr C_p}=U_\beta(I_u)\subset U_\beta(J_0).$ 
\smallskip

\noindent {\bf{(ii).}} If none of the fundamental intervals are contained in $J_0,$ then we have 
the following possibilities:  
\smallskip

\noindent {\bf{(a).}} Suppose $J_0$ is contained in one of the fundamental intervals namely, 
$I_u,~u\in\mathcal U,$ then by uniform expansiveness condition of $U_\beta,$ we have 
$m(J_1)\geq\beta_0m(J_0)\geq\frac{\beta_0}{2}m(J_0),$ where $J_1:=U_\beta(J_0).$
\smallskip

\noindent {\bf{(b).}} Suppose $J_0$ has nonempty intersection with two neighbouring fundamental intervals 
say $I_u,I_{u'}$ and $J_0$ is contained in the closure of ${I_u\cup I_{u'}}.$ In this case, one of the sets 
$J_0\cap I_u,~J_0\cap I_{u'}$ say, $J_0\cap I_u$ has length at least $\frac{1}{2}m(J_0).$ It follows that for  $J_1:=U_\beta(J_0\cap I_u)$ with $J_1\subset U_\beta(J_0),$ we have 
\[m(J_1)=m\Big(U_\beta(J_0\cap I_u)\Big)\geq\beta_0m(J_0\cap I_u)\geq\frac{\beta_0}{2}m(J_0).\]
For both the cases {\bf{(a)}} and {\bf{(b)}}, there exist an interval $J_1$ which is contained in 
$U_\beta(J_0)$ and $m(J_1)\geq\frac{\beta_0}{2}m(J_0).$ Next, we consider $J_1$ in place of $J_0,$ and 
therefore we get an interval namely, $J_2$ such that $m(J_2)\geq\frac{\beta_0}{2}m(J_1)\geq\big(\frac{\beta_0}{2}\big)^2m(J_0).$ We repeat this process, and hence we get an increasing sequence of intervals namely, $J_0,J_1,J_2,\ldots$ such that $m(J_l)\geq\big(\frac{\beta_0}{2}\big)^lm(J_0)$ with $J_l\subseteq U_\beta^l(J_0).$ Since the length of $J_l$ depends on $l,$ after finitely many steps say, $l=l_0$ we must stop this process, 
because $J_{l_0}$ will contain one of the fundamental intervals, and hence ${\mathscr C_p}\subset U_\beta(J_{l_0})\subseteq U_\beta^{l_0+1}(J_0).$

\smallskip

\noindent{\bf\textit{2. The case $\beta_0$ is not an odd integer.}} 
\smallskip

\noindent Since $\{\beta_0\}_2\in(-1,1),$ it follows that $2u_0-1<\beta_0<2u_0+1.$ The fundamental 
intervals $I_u$ are given by (\ref{eq199}) for all $u\in\mathcal U\setminus\{\pm u_0\},$ where 
$\mathcal U$ be the set of all nonzero integers with $|u|\geq u_0.$ The edge fundamental intervals 
$I_{u_0},~I_{-u_0}$ are given by (\ref{eq200}). For every $u\in\mathcal U,$ the map $U_\beta$ is given 
by $U_\beta(x)=p(2u-\beta/x)$ for all $x\in I_u.$
\smallskip

\noindent{\bf\textit{2(A). The case $J_0$ is an edge fundamental interval.}}  
Assume that $J_0=I_{-u_0},$ and the case $J_0=I_{u_0}$ is similar. In this case, we want to show that 
${\mathscr C_p}\subset U_\beta^n(J_0)$ for sufficiently large $n.$ We first observe that 
\begin{equation}\label{eq2020}
U_\beta(J_0)=U_\beta(I_{-u_0})=\Big(p(\beta_0-2u_0),~p\Big)\supset I_{u_0}':=\left(p(\beta_0-2u_0),~\frac{\beta}{2u_0+1}\right).
\end{equation}
Then there are two possibilities: 
\smallskip

\noindent {\bf{(i).}} If $(\beta_0-2u_0)\leq\beta_0/(2u_0+3),$ then we have $I_{u_0+1}\subset I_{u_0}'.$ 
Therefore, ${\mathscr C_p}=U_\beta(I_{u_0+1})\subset U_\beta(I_{u_0}')\subset U_\beta^2(I_{-u_0})=U_\beta^2(J_0).$
\smallskip

\noindent {\bf{(ii).}} If $(\beta_0-2u_0)>\beta_0/(2u_0+3),$ then $I_{u_0}'\subset I_{u_0+1}$ and the point 
$p(\beta_0-2u_0)\in I_{u_0+1}.$ Next, we claim that for $y\in I_{u_0}'':=\left(\frac{\beta}{2u_0+3},~p(\beta_0-2u_0)\right]\subset I_{u_0+1},$ there exists a positive number $\beta_0'$ with $\beta_0\geq\beta_0'>1,$ 
which depends only on $\beta_0$ and closest to the point $1$ such that \[m\Big(U_\beta(I_y)\Big)\geq\beta_0'm\Big((y,~p)\Big),\] where $I_y:=\left(y,~\frac{\beta}{2u_0+1}\right).$ Since $U_\beta(I_y)=\Big(p(2u_0+2-\beta/y),~1\Big),$ it is enough to show that for $y\in I_{u_0}'',$ 
\begin{equation}\label{eq203}
\beta_0'y+p\beta/y\geq p(2u_0+1)+p\beta_0'.
\end{equation}
By the change of variables $y:=py',$ it is equivalent to show that for $y'\in I_{u_0}''':=\left(\frac{\beta_0}{2u_0+3},~(\beta_0-2u_0)\right],$ we have 
\begin{equation}\label{eq204}
\beta_0' y'+\beta_0/{y'}\geq 2u_0+1+\beta_0'.
\end{equation}
Now, (\ref{eq204}) is same as (\cite{MHR}, p. 54, Equation 7.4). Thus we conclude that $\beta_0'$ 
exists and sufficiently close to 1 so that the minimum exists in (\ref{eq203}) at $y_0:=p(\beta_0-2u_0).$ 
Note that $\beta_0>2$ because $y_0>\beta_0/(2u_0+3).$ In particular, we have 
$m\Big(\left(U_\beta(y_0),~p\right)\Big)\geq\beta_0'm\Big((y_0,~p)\Big).$ 
If we consider $J_1:=U_\beta(J_0)=(y_0,~p)$ and $J_2:=U_\beta(I_{u_0}')=\Big(U_\beta(y_0),~p\Big),$ 
then $J_2\subset U_\beta^2(J_0)=U_\beta(J_1)$ with $m\left(J_2\right)\geq\beta_0'm\left(J_1\right).$ If 
$U_\beta(y_0)\leq\beta/(2u_0+3),$ then $I_{u_0+1}\subset J_2,$ and hence we are done, because 
${\mathscr C_p}=U_\beta(I_{u_0+1})\subset U_\beta(J_2)=U_\beta^2(I_{u_0}')\subset U_\beta^3(I_{-u_0})=U_\beta^3(J_0).$ If $U_\beta(y_0)>\beta/(2u_0+3),$ then repeat the same argument to get a bigger interval say, $J_3$ 
with right end point $1$ so that $I_{u_0+1}$ is contained in $J_3.$ This completes proof of the case 
$J_0=I_{-u_0}.$
\smallskip
 
\noindent{\bf\textit{2(B). The case $J_0\subset[-p,p]$ is an arbitrary nonempty open interval.}} 
Recall that $2u_0-1<\beta_0<2u_0+1$ and the edge fundamental intervals are given by (\ref{eq200}). 
Let the point $x_0\in I_{u_0}$ is given by 
\begin{equation}\label{eq201}
x_0:=\frac{(2{u_0}+1)\beta}{2u_0(2{u_0}+1)+\beta_0}.
\end{equation}
A simple calculation gives $p\beta/x_0^2>2.$ Since on the fundamental interval $I_u$ the Gauss-type map is 
given by $U_\beta(x)=2u-\beta/x,$ the point $x_0\in I_{u_0}$ has the property that 
\begin{equation}\label{eq202}
U_\beta\left(I_{x_0}\right)=I_{-u_0}~\text{and}~U_\beta\left(I_{-x_0}\right)=I_{u_0}, 
\end{equation}
where  
\[I_{x_0}:=\left(\frac{\beta}{2{u_0}+1},~x_0\right)~\text{and}~I_{-x_0}:=\left(-x_0,~-\frac{\beta}{2{u_0}+1}\right).\]
Moreover, for $x\in[-x_0,x_0]\cap\bigcup\{I_u:~u\in\mathcal U\},$ we have $U_\beta'(x)\geq p\beta/x_0^2>2.$ 
Therefore, if we write $\beta_0'':=\text{min}\Big\{\beta_0,~\frac{p\beta}{2x_0^2}\Big\},$ then $\beta_0''>1.$  
In this case, we show that ${\mathscr C_p}\subset U_\beta^n(J_0)$ for sufficiently large $n.$ 
There are two possibilities: 
\smallskip

\noindent {\bf{(i).}} Suppose $J_0$ contains one of the fundamental intervals say, $I_u$ for some 
$u\in\mathcal U\setminus\{\pm u_0\},$ then it directly follows that ${\mathscr C_p}=U_\beta(I_u)\subset U_\beta(J_0).$ If $J_0$ contains the edge fundamental intervals, then also we are done by the case 
{\bf{2(A)}} above.
\smallskip

\noindent {\bf{(ii).}} If none of the fundamental intervals are contained in $J_0,$ then we have 
the following possibilities:  
\smallskip

\noindent {\bf{(a).}} Suppose $J_0\subset I_u$ for some $u\in\mathcal U,$ then by uniform expansiveness condition 
of $U_\beta,$ we get that $m(J_1)\geq\beta_0m(J_0)\geq\beta_0''m(J_0),$ where $J_1:=U_\beta(J_0).$
\smallskip

\noindent {\bf{(b).}} Suppose $J_0$ has nonempty intersection with two neighbouring fundamental intervals 
say $I_u,I_{u'}$ and $J_0$ is contained in the closure of ${I_u\cup I_{u'}}.$ There are two possibilities: 
\smallskip

\noindent {\bf{(b1).}} Assume that $J_0\subset[-x_0,x_0].$ In this case, one of the sets $J_0\cap I_u,~J_0\cap I_{u'}$ say, $J_0\cap I_u$ has length at least $\frac{1}{2}m(J_0).$ It follows that for $J_1:=U_\beta(J_0\cap I_u)$ 
with $J_1\subset U_\beta(J_0),$ 
\[m(J_1)=m\Big(U_\beta(J_0\cap I_u)\Big)\geq\frac{p\beta}{x_0^2}m(J_0\cap I_u)\geq\frac{p\beta}{2x_0^2}m(J_0)
\geq\beta_0''m(J_0).\]
For both the cases {\bf{(a)}} and {\bf{(b1)}}, there exist an interval $J_1$ which is contained in 
$U_\beta(J_0)$ with $m(J_1)\geq\beta_0''m(J_0).$ Next, we consider $J_1$ in place of $J_0,$ and therefore we 
get a bigger interval namely, $J_2$ such that $m(J_2)\geq\beta_0''m(J_1)\geq{\beta_0''}^2m(J_0).$ We 
repeat this process and hence we get an increasing sequence of intervals namely, $J_0,J_1,J_2,\ldots$ 
such that $m(J_l)\geq{\beta_0''}^lm(J_0)$ with $J_l\subseteq U_\beta^l(J_0).$ Since the length of $J_l$ 
depends on $l,$ after finitely many steps say, $l=l_0$ we must stop this process, because $J_{l_0}$ will contain one 
of the fundamental intervals, and hence ${\mathscr C_p}\subset U_\beta(J_{l_0})\subseteq U_\beta^{l_0+1}(J_0).$
\smallskip

\noindent {\bf{(b2).}} It only remains the case when $J_0$ is not contained in $[-x_0,x_0].$ Then 
$\bar I_{x_0}\subset J_0\cap I_{u_0}$ or $\bar I_{-x_0}\subset J_0\cap I_{-u_0},$ and hence we are done, 
because in view of (\ref{eq202}) we have one of the following: 
\begin{enumerate}[(\text{b2}1).]
\item ${\mathscr C_p}=U_\beta(I_{-u_0})=U_\beta^2(I_{x_0})\subset U_\beta^2(J_0\cap I_{u_0})\subset U_\beta^2(J_0),$
\smallskip

\item ${\mathscr C_p}=U_\beta(I_{u_0})=U_\beta^2(I_{-x_0})\subset U_\beta^2(J_0\cap I_{-u_0})\subset U_\beta^2(J_0).$
\end{enumerate}
This completes the proof of Lemma \ref{lemma10}.
\end{proof}

\subsection{Characterization of the pre-annihilator space $\mathcal F_\beta^\perp$}

\subsubsection{} \textit{Periodic and inverted periodic functions.}
\smallskip

Let $L^\infty_{p}(\mathbb R)$ denote the space of all functions $f\in L^\infty(\mathbb R)$ 
such that the map $x\longmapsto e^{-\pi ix/p}f(x)$ is 2-periodic. Then the weak-star closure 
in $L^\infty(\mathbb R)$ of the linear span of the functions $\{e^{p}_n(x):=e^{\pi i(n+1/p)x};~n\in\mathbb Z\}$ equals to $L^\infty_{{p}}(\mathbb R).$
\smallskip

Let $L^\infty_{\beta}(\mathbb R)$ denote the space of all functions $f\in L^\infty(\mathbb R)$ 
such that the map $x\longmapsto f(\beta/x)$ is 2-periodic. Then the weak-star closure in 
$L^\infty(\mathbb R)$ of the linear span of the functions $\{e^{\beta}_n(x):=e^{\pi in\beta/x};~n\in\mathbb Z\}$ equals to $L^\infty_{\beta}(\mathbb R).$ 
\smallskip

Observe that the functions in $L^\infty_{{p}}(\mathbb R)$ are defined freely on $[-p,p],$ 
and due to periodicity they are uniquely determined on $\mathbb R\setminus[-p,p].$ 
Similarly, the functions in $L^\infty_{\beta}(\mathbb R)$ are defined freely on 
$\mathbb R\setminus[-\beta,\beta],$ and due to periodicity they are uniquely determined 
on $[-\beta,\beta].$ For $E\subseteq\mathbb R,$ the function $\chi_E$ denote the 
characteristic function of $E$ on $\mathbb R.$ This observations motivate to define 
the operators $\text{S}_p,\text{T}_{\beta}$ as follows.
\smallskip

The operator $\text{S}_p : L^\infty([-p,p])\rightarrow L^\infty(\mathbb R\setminus[-p,p])$ 
is defined by   
\[\text{S}_p[\varphi](x)=\varphi\left(p\left\{x/p\right\}_2\right)\chi_{\mathbb R\setminus[-p,p]}(x),\]
where $\varphi\in L^\infty([-p,p]).$ The operator 
$\text{T}_{\beta} : L^\infty(\mathbb R\setminus[-\beta,\beta])\rightarrow L^\infty([-\beta,\beta])$ 
is defined by  
\[\text{T}_{\beta}[\psi](x)=\psi\left(\frac{\beta}{\{\beta/x\}_2}\right)\chi_{[-\beta,\beta]\setminus\{0\}}(x),\] 
where $\psi\in L^\infty(\mathbb R\setminus[-\beta,\beta]).$ Now, in terms of the operators 
$\text{S}_p,\text{T}_{\beta}$ the functions space $L^\infty_{p}(\mathbb R)$ and $L^\infty_{\beta}(\mathbb R)$ 
are given by 
\[\begin{cases} 
 L^\infty_{p}(\mathbb R)=\Big\{\varphi+\text{S}_p[\varphi]:\varphi\in L^\infty([-p,p])\Big\},\\
 L^\infty_{\beta}(\mathbb R)=\Big\{\psi+\text{T}_{\beta}[\psi]:\psi\in L^\infty(\mathbb R\setminus[-\beta,\beta])\Big\}.
\end{cases}
\]

\subsubsection{} \textit{The Perron-Frobenius operator.}
For $p<\beta<\infty,$ the \textit{Koopman operator} $\mathcal C_\beta : L^\infty([-p,p])\rightarrow L^\infty([-p,p])$ associated to $U_\beta$ be the map \[\mathcal C_\beta[\varphi](x)=\varphi\circ U_\beta(x), 
~~x\in[-p,p].\] The predual adjoint of $\mathcal C_\beta$ is the \textit{Perron-Frobenius}  
operator $\mathcal P_\beta : L^1([-p,p])\rightarrow L^1([-p,p])$ associated to $U_\beta$ is given by 
\[\mathcal P_\beta [h](x)=\sum\limits_{u\in\mathbb Z^\ast}\frac{p\beta}{(2pu-x)^2}h\left(\frac{p\beta}{2pu-x}\right),~~x\in[-p,p].\]
The operator $\mathcal P_\beta$ is linear and a norm contraction on $L^1([-p,p]).$ Thus the point spectrum $\sigma_{point}(\mathcal P_\beta)$ of $\mathcal P_\beta$ is contained in the closed unit disk $\bar{\mathbb D}.$ Here, we use the notation $\mathcal C_\beta$ in place of 
$\mathcal C_{U_\beta}$ and $\mathcal P_\beta$ for $\mathcal P_{U_\beta}.$
\smallskip

Next, we build up a connection between the operators $\text{S}_p,\text{T}_{\beta}$ and $\mathcal C_\beta,$  
so that we can study the pre-annihilator space $\mathcal F_\beta^\perp$ via the properties of 
the Perron-Frobenius operators for $p<\beta<\infty.$ To do this, we need to define some restriction 
operators. For a measurable set $E\subseteq\mathbb R$ with $m(E)>0,$ we denote $L^s(E)$ the closed 
subspace of $L^s(\mathbb R)$ by extending the functions vanish on $\mathbb R\setminus E,$ where 
$s=1,\infty.$ For $p<\beta<\infty,$ consider the following restriction operators: 
\[\begin{cases} 
\text{R}_1~:~L^\infty(\mathbb R\setminus[-p,p])\rightarrow 
L^\infty(\mathbb R\setminus[-\beta,\beta]), \\
\text{R}_2~:~L^\infty([-\beta,\beta])\rightarrow 
L^\infty([-p,p]),\\
\text{R}_3~:~L^\infty([-\beta,\beta])\rightarrow 
L^\infty([-\beta,\beta]\setminus[-p,p]),\\
\text{R}_4~:~L^\infty(\mathbb R\setminus[-p,p])\rightarrow 
L^\infty([-\beta,\beta]\setminus[-p,p]). 
   \end{cases}
\] 
Then the pre-dual adjoints (Banach space dual) are the maps $\text{R}_1^*,\text{R}_2^*,\text{R}_3^*$ 
and $\text{R}_4^*$ defined on the corresponding $L^1$-spaces. A simple calculation shows that $\mathcal C_\beta^2=\text{R}_2\text{T}_{\beta}\text{R}_1\text{S}_p$ whenever $\beta>p,$ and hence $\mathcal P_\beta^2=\text{S}_p^*\text{R}_1^*\text{T}_{\beta}^*\text{R}_2^*.$ 

\begin{proposition}\label{prop30}
For $p<\beta<\infty,$ suppose $f\in L^1(\mathbb R)$ such that $f=f_1+f_2+f_3,$ where 
$f_1\in L^1([-p,p]),$ $f_2\in L^1([-\beta,\beta]\setminus[-p,p]),$ and 
$f_3\in L^1(\mathbb R\setminus[-\beta,\beta]).$ Then $f\in\mathcal F_\beta^\perp$ 
if and only if $(i)~~(\text{I}-\mathcal P_\beta^2)f_1=\text{S}_p^*(-\text{R}_4^*+\text{R}_1^*\text{T}_{\beta}^*\text{R}_3^*)f_2,$ where $\text{I}$ is the identity operator on $L^1([-p,p]),$ and $(ii)~~f_3=-\text{T}_{\beta}^*\text{R}_2^*f_1-\text{T}_{\beta}^*\text{R}_3^*f_2.$ 
\end{proposition}

\begin{proof}
The proof of Proposition \ref{prop30} works along the same lines as in the proof of 
(\cite{MHR}, p. 52, Proposition 6.1), hence omitted.
\end{proof}

\begin{remark}[$0<\beta\leq p$]
The composition operator $\text{T}_{\beta}\text{S}_p : L^\infty([-p,p])\rightarrow 
L^\infty([-p,p])$ is given by    
\[\text{T}_{\beta}\text{S}_p[\varphi](x)=\varphi\left({p}\left\{\frac{\beta_0}{\{\beta/x\}_2}\right\}_2\right)\chi_{E_{\beta}}(x),~~\varphi\in L^\infty([-p,p]),\] where $\beta_0=\beta/{p}$ and 
$E_{\beta}=\left\{x\in(-\beta,\beta]\setminus\{0\}: \frac{\beta_0}{\{\beta/x\}_2}\in\mathbb R\setminus(-1,1]\right\},$
and the weighted Koopman operator $\mathcal C_\beta : L^\infty([-p,p])\rightarrow L^\infty([-p,p])$ associated to $U_\beta$ be the map $\mathcal C_\beta[\varphi](x)=\varphi\circ U_\beta(x)\chi_{[-\beta,\beta]}(x),$  
where $x\in\mathbb R.$ The predual adjoint of $\mathcal C_\beta$ is 
the Perron-Frobenius operator $\mathcal P_\beta : L^1([-p,p])\rightarrow L^1([-p,p])$ is given by 
$\mathcal P_\beta [h](x)=\sum\limits_{u\in\mathbb Z^\ast}\frac{p\beta}{(2pu-x)^2}h\left(\frac{p\beta}{2pu-x}\right).$  As the operator $\mathcal P_\beta$ is linear and a norm contraction, the point 
spectrum $\sigma_{point}(\mathcal P_\beta)$ of $\mathcal P_\beta$ is contained in $\bar{\mathbb D}.$  
Observe that $\text{T}_{\beta}\text{S}_p=\mathcal C_\beta^2.$ In \cite{GRA}, it has shown that 
$\lambda\in\partial\bar{\mathbb D}$ is not an eigenvalue of $\mathcal P_\beta$ whenever $0<\beta\leq p,$  
which in turn implies that $\mathcal{AC}\left(\Gamma,\Lambda_\beta^\theta\right)=\{0\}$ for $0<\beta_0\leq 1.$
\end{remark}

\subsection{Exterior spectrum of $\mathcal P_\beta,~\beta>p$ and the proof of Theorem \ref{th8}.}

Next, we study the exterior spectrum of the Perron-Frobenius operator $\mathcal P_\beta$ for $p<\beta<\infty.$ 
We know from Theorem A that $1$ is an eigenvalue of $\mathcal P_\beta$ and the associated eigenfunction is 
in $\text{BV}([-p,p]).$ The proof of Theorem \ref{th13} gives us that one of the eigenfunctions for the 
eigenvalue $1$ must be positive, and it can be normalized by a suitable constant so that we get the positive 
density of an ergodic $U_\beta$-invariant absolutely continuous probability measure.

\begin{theorem}\label{th13}
Let $p<\beta<\infty,$ then $\alpha_1=1$ is a simple eigenvalue of $\mathcal P_\beta,$ and is 
the only eigenvalue of $\mathcal P_\beta$ contained in $\partial\bar{\mathbb D}.$ Moreover, the 
eigenfunctions for $\alpha_1=1$ are nonzero scalar multiple of $\varrho_0,$ 
where $\varrho_0dm$ is the unique ergodic $U_\beta$-invariant absolutely continuous 
probability measure with $\varrho_0>0$ almost everywhere. 
\end{theorem}

\begin{proof} 
Here, Lemma \ref{lemma10} will help to show that $\varrho_0>0$ almost everywhere. The proof of 
Theorem \ref{th13} works along the same lines as in the proof of (\cite{MHR}, p. 55, Theorem 7.2), 
hence omitted.
\end{proof}
 
For $p<\beta<\infty,$ the Gauss-type map $\tilde U_\beta : [-\beta,\beta]\rightarrow[-\beta,\beta]$ is 
defined by letting 
\[\tilde U_\beta(x):=\begin{cases} 
      p\left\{-\frac{\beta}{x}\right\}_2,~x\neq0, \\
      \quad 0 \qquad\text{for}~ x=0. 
   \end{cases}\]
The \textit{Koopman operator} $\tilde{\mathcal C}_\beta : L^\infty([-\beta,\beta])\rightarrow L^\infty([-\beta,\beta])$ associated to $\tilde U_\beta$ be the map $\tilde{\mathcal C}_\beta[\varphi](x)=\varphi\circ \tilde U_\beta(x),~~x\in[-\beta,\beta].$ Then the pre-dual adjoint of $\tilde{\mathcal C}_\beta$ is the \textit{Perron-Frobenius} operator $\tilde{\mathcal P}_\beta : L^1([-\beta,\beta])\rightarrow L^1([-\beta,\beta])$ associated to $\tilde U_\beta.$ A simple calculation shows that $\tilde{\mathcal C}_\beta^2=\text{T}_{\beta}\text{R}_1\text{S}_p\text{R}_2$ whenever $\beta>p,$ and hence $\tilde{\mathcal P}_\beta^2=\text{R}_2^*\text{S}_p^*\text{R}_1^*\text{T}_{\beta}^*.$ It is easy to see that $\tilde U_\beta$ is a \textit{partially filling $\mathcal C^2$-smooth piecewise monotonic transform} for $\beta>p.$ 
Also, $\tilde U_\beta$ satisfy the \textit{uniform expansiveness} condition with $m=2,$ and \textit{second derivative condition}. Therefore, it follows from Theorem A that $\tilde{\mathcal P}_\beta$ maps $\text{BV}([-\beta,\beta])$ into $\text{BV}([-\beta,\beta]).$ Now, following the proof of (\cite{MHR}, Lemma 8.1), for any $f\in\text{BV}([-\beta,\beta]\setminus[-p,p]),$ we have 
\begin{equation}\label{eq210}
-\text{S}_p^*\text{R}_4^*f+\text{S}_p^*\text{R}_1^*\text{T}_{\beta}^*\text{R}_3^*f\in\text{BV}([-p,p]).
\end{equation}
To complete the proof of Theorem \ref{th8}, we prove the following result. Although the proof 
of Theorem \ref{th14} follows the same lines as in (\cite{MHR}, Theorem 8.2), we write it here 
for the sake of completeness.

\begin{theorem}\label{th14}
For $p<\beta<\infty,$ there exits a bounded linear operator $\mathcal B~:~\text{BV}([-\beta,\beta]\setminus[-p,p])\rightarrow L^1(\mathbb R)$ such that the range of $\mathcal B$ is infinite-dimensional, and 
contained in $\mathcal F_\beta^\perp.$ Moreover, the range of $\mathcal B$ is contained in the 
weighted $L^2$-space $L^2(\mathbb R,\omega),$ where $\omega(x)=1+x^2.$
\end{theorem}

\begin{proof}
We actually prove more precise statement, namely that, there exits a bounded linear operator 
$\mathcal B~:~\text{BV}([-\beta,\beta]\setminus[-p,p])\rightarrow L^1(\mathbb R)$ such that   
$\mathcal Bf(x)=f(x)$ a.e. $x\in [-\beta,\beta]\setminus[-{p},{p}],$ and for all 
$f\in\text{BV}([-\beta,\beta]\setminus[-p,p]).$ Moreover, $\mathcal B$ has infinite-dimensional 
range which is contained in $\mathcal F_\beta^\perp.$  
\smallskip

In view of Theorem A and Theorem \ref{th13}, we have the following spectral decomposition for the 
Perron-Frobenius operator $\mathcal P_\beta$ associated to $U_\beta :$ 
\begin{equation}\label{eq20}
\mathcal P_\beta^n[h]=\{\langle h,\phi_0\rangle_{[-p,p]}\}\varrho_0+\mathcal Z_\beta^n[h],~n=1,2,\ldots,
\end{equation} 
where $h\in L^1([-p,p]),$ and $\phi_0\in L^\infty([-p,p])$ such that $\langle \varrho_0,\phi_0\rangle_{[-p,p]}=1.$  
Also, $\varrho_0$ is the positive density of the ergodic $U_\beta$-invariant absolutely continuous 
probability measure on $[-p,p],$ and we have $\varrho_0\in\text{BV}([-p,p]),$ in addition, we can 
normalized by a suitable constant so that $\langle\varrho_0,1\rangle_{[-p,p]}=1.$ Moreover, $\mathcal Z_\beta$ 
acts on $\text{BV}([-p,p])$ and its spectral radius smaller than $1.$ In particular, we have 
$\mathcal Z_\beta[\varrho_0]=0,$ because $\varrho_0$ is invariant under $\mathcal P_\beta.$ 
Observe that, $\mathcal Z_\beta^n[h]\rightarrow0$ exponentially as $n\rightarrow\infty.$
\smallskip
 
Next, we claim that $\phi_0$ is the constant function $1$ almost everywhere on $[-p,p].$ To see 
this, let $h\in\text{BV}([-p,p]),$ then by (\ref{eq20}) we infer that  
\begin{eqnarray}\label{eq21}
\langle h,1\rangle_{[-p,p]} &=& \langle h,\mathcal C_\beta^n[1]\rangle_{[-p,p]}
=\langle \mathcal P_\beta^n[h],1\rangle_{[-p,p]}\\
&=& \langle h,\phi_0\rangle_{[-p,p]}\langle\varrho_0,1\rangle_{[-p,p]}+\langle\mathcal Z_\beta^n[h],1\rangle_{[-p,p]}\nonumber\\
&=& \langle h,\phi_0\rangle_{[-p,p]}+\langle\mathcal Z_\beta^n[h],1\rangle_{[-p,p]}\rightarrow
\langle h,\phi_0\rangle_{[-p,p]},~~\text{as}~n\rightarrow\infty.\nonumber
\end{eqnarray}
It is well known that $\text{BV}([-p,p])$ is dense in $L^1([-p,p]).$ Thus from (\ref{eq21}) get the claim. 
Further, as soon as $\phi_0=1,$ (\ref{eq20}) can be rewrite as 
\begin{equation}\label{eq22}
\mathcal P_\beta^n[h]=\{\langle h,1\rangle_{[-p,p]}\}\varrho_0+\mathcal Z_\beta^n[h],~n=1,2,\ldots, 
\end{equation}
where $h\in L^1([-p,p]),$ and from (\ref{eq21}) we get that 
\begin{equation}\label{eq23}
\langle\mathcal Z_\beta^n[h],1\rangle_{[-p,p]}=0;~~h\in L^1([-p,p]),~n=1,2,\ldots. 
\end{equation}
\smallskip

Now, we are in a position to construct an extension operator $\mathcal B$ from $\text{BV}([-\beta,\beta]\setminus[-p,p])$ onto $L^1(\mathbb R).$ To do so, pick an arbitrary $f_2\in\text{BV}([-\beta,\beta]\setminus[-p,p]),$ then from (\ref{eq210}), we know that $-\text{S}_p^*\text{R}_4^*f_2+\text{S}_p^*\text{R}_1^*\text{T}_{\beta}^*\text{R}_3^*f_2\in\text{BV}([-p,p]).$ Since $\mathcal Z_\beta$ acts on $\text{BV}([-p,p])$ has spectral radius 
smaller than $1,$ $I-\mathcal Z_\beta^2$ is invertible and hence, we define the operator $\mathcal B_1$ by letting  
\begin{equation}\label{eq24}
\mathcal B_1f_2 :=(I-\mathcal Z_\beta^2)^{-1}(-\text{S}_p^*\text{R}_4^*+\text{S}_p^*\text{R}_1^*\text{T}_{\beta}^*\text{R}_3^*)f_2\in\text{BV}([-p,p]).
\end{equation}
A simple calculation gives $\langle{-\text{S}_p^*\text{R}_4^*f_2+\text{S}_p^*\text{R}_1^*\text{T}_{\beta}^*\text{R}_3^*f_2},1\rangle_{[-p,p]}=0.$ If we denote $f_1:=\mathcal B_1f_2,$ then by (\ref{eq23}) and (\ref{eq24}), 
\begin{equation}\label{eq25}
\langle f_1,1\rangle_{[-p,p]}=\left\langle {\left(I-\mathcal Z_\beta^2\right)}[f_1],1\right\rangle_{[-p,p]}=0.
\end{equation}
Next, we define the operator $\mathcal B_3$ on $\text{BV}([-\beta,\beta]\setminus[-p,p])$ by letting 
\begin{equation}\label{eq26}
\mathcal B_3f_2 :=-\text{T}_{\beta}^*\text{R}_2^*f_1-\text{T}_{\beta}^*\text{R}_3^*f_2\in L^1(\mathbb R\setminus[-\beta,\beta]).
\end{equation}
We write $f_3:=\mathcal B_3f_2.$ Now, we define the operator $\mathcal B:\text{BV}([-\beta,\beta]\setminus[-p,p])\rightarrow L^1(\mathbb R)$ by letting \[\mathcal Bf_2 :=f_1+f_2+f_3\in L^1(\mathbb R),\]
with the understanding that each $f_k;~k=1,2,3$ can be extended to $\mathbb R$ by considering zero outside of their domain of definition. Then the bounded and linear operator $\mathcal B$ is clearly an extension operator, in the sense that for all  $f_2\in \text{BV}([-\beta,\beta]\setminus[-p,p]),$ 
\[\mathcal Bf_2(x)=f_2(x);~~\text{a.e.}~x\in [-\beta,\beta]\setminus[-p,p].\] 
Observe that the range of $\mathcal B$ is infinite-dimensional.  
Next, we claim that the range of $\mathcal B$ is contained in $\mathcal F_\beta^\perp.$ Actually, we have to 
verify the conditions $(i)$ and $(ii)$ of the Proposition \ref{prop30} for the functions $f_k;~k=1,2,3.$ 
In view of (\ref{eq25}), from (\ref{eq22}) we get that $\mathcal P_\beta^n[f_1]=\mathcal Z_\beta^n[f_1];~n=1,2,\ldots.$ Thus, from (\ref{eq24}) and (\ref{eq26}) we have the conditions $(i)$ and $(ii)$ of the Proposition \ref{prop30}.   
\smallskip

It follows from the proof of (Proposition 8.3, \cite{MHR}) that the range of $\mathcal B$ is contained in the weighted $L^2$-space $L^2(\mathbb R,\omega),$ where the weight $\omega(x)=1+x^2.$ This completes the proof of Theorem \ref{th14}.
\end{proof}

\section{Proof of Theorem \ref{th12}}

\subsection{Dynamics of a Gauss-type map}

\subsubsection{} \textit{A Gauss-type map.} For $t\in\mathbb R,$ the expression $\{t\}_1$ is the 
unique number in $[0,1)$ such that $t-\{t\}_1\in\mathbb Z.$ For $0<\gamma<\infty,$ consider the  
Gauss-type map $V_\gamma$ on the interval $[0,q).$ The map $V_\gamma : [0,q)\rightarrow[0,q)$ 
is defined by letting 

\[V_\gamma(x)=\begin{cases} 
      q\left\{\frac{\gamma}{x}\right\}_1,~x\neq0 \\
      \quad 0 \qquad x=0. 
   \end{cases}
\] 

Note that, for $v\in\mathbb N,$ the map $V_\gamma$ can be expressed as $V_\gamma(x)=q\left(\frac{\gamma}{x}-v\right)$ whenever $\frac{\gamma}{v+1}<x\leq\frac{\gamma}{v},$ and hence $V_\gamma : \left(\frac{\gamma}{v+1},\frac{\gamma}
{v}\right]\rightarrow[0,q)$ is one-to-one.

\subsubsection{} \textit{Dynamical properties of Gauss-type map for $\gamma>q.$}
\smallskip 

Denote $\gamma_0:=\gamma/q$ which is $>1.$ In this section, we observe that $V_\gamma$ is a \textit{partially 
filling $\mathcal C^2$-smooth piecewise monotonic transform} for $\gamma_0>1.$ We need to find a 
unique absolutely continuous invariant probability measure for $V_\gamma$ having positive density.
\smallskip

Let $\mathcal V:=\mathcal V_\gamma$ denote the index set which contain the points $v\in\mathbb N$ 
so that the associate fundamental interval is nonempty, that is, 
\[J_v :=\left(\frac{\gamma}{v+1},\frac{\gamma}{v}\right)\cap(0,q)\neq \emptyset.\]   
If $\gamma_0$ is an integer, then it is easy to see that 
for all $v\in\mathcal V,$ where $\mathcal V$ be the set of all nonzero positive integers with 
$v\geq\gamma_0,$ the fundamental intervals are given by  
$J_v :=\left(\frac{\gamma}{v+1},\frac{\gamma}{v}\right),~~v\in\mathcal V$  
so that $V_\gamma(J_v)=(0,q)$ for all $v\in\mathcal V.$ Thus in this particular case,  
$V_\gamma$ fulfill the \text{"filing"} condition for all $v\in\mathcal V.$  
If $\gamma_0$ is not an integer, write $v_0 :=\gamma_0-\{\gamma_0\}_1\geq 1,$ then a 
simple calculation shows that 
$J_v :=\left(\frac{\gamma}{v+1},\frac{\gamma}{v}\right),~~v\in\mathcal V\setminus\{v_0\},$ 
where $\mathcal V$ be the set of all positive integers with $v\geq v_0.$ 
Observe that $V_\gamma(J_v)=(0,q)$ for all $v\in\mathcal V\setminus\{v_0\},$ that is, 
in this case, $V_\gamma$ fulfill the \text{"filing"} condition for all $v\in\mathcal V$ 
except one branch corresponding to $\{v_0\}.$ In view of the above facts, we conclude that 
$V_\gamma$ is a \textit{partially filling $\mathcal C^2$-smooth piecewise monotonic transform} 
for $\gamma_0>1.$
\smallskip

The map $V_\gamma$ satisfy the \textit{"uniform expansiveness"} condition with $m=1,$ because  
of the derivative $|V_\gamma'(x)|=\frac{q\gamma}{x^2}\geq \gamma_0>1$ for all $x\in\{J_v : v\in\mathcal V\}.$   
Also, we have \textit{"second derivative condition"} due to $|V_\gamma''(x)|\leq\frac{2}{q}|V_\gamma'(x)|^2$ 
for all $x\in\{J_v : v\in\mathcal V\}.$ We aim to find a unique absolutely continuous invariant 
probability measure corresponding to $V_\gamma$ which has a positive density. Since $V_\gamma$ is a \textit{partially filling $\mathcal C^2$-smooth piecewise monotonic transform},  
there exists a $V_\gamma$-invariant absolutely continuous probability measure, but it may not be unique. 
Although, form (\cite{MHR}, p. 45, Remark 5.2) we can get the uniqueness for \textit{filling $\mathcal C^2$
-smooth piecewise monotonic transform}. Also, $V_\gamma$ may not always be a Markov map, therefore, 
to get the uniqueness of the absolutely continuous $V_\gamma$-invariant measure, we have to prove 
the condition $(iii)$ of Adler's theorem which is stated in (\cite{MHR}, p. 46, Theorem D) for the details 
see \cite{BG, LY}.

\subsection{Characterization of the pre-annihilator space $\mathcal K_\gamma^\perp$}

\subsubsection{} \textit{Periodic and inverted periodic functions.}
\smallskip 

Let $L^\infty_{q}(\mathbb R_+)$ denote the space of all functions $f\in L^\infty(\mathbb R_+)$ such that the map $x\longmapsto e^{-2\pi ix/q}f(x)$ is 1-periodic. Then the weak-star closure in $L^\infty(\mathbb R_+)$ of the linear span of the functions $\{e^{q}_n(x):=e^{2\pi i(n+1/q)x};~n\in\mathbb Z\}$ equals to $L^\infty_{q}(\mathbb R_+).$  
Let $L^\infty_{\gamma}(\mathbb R_+)$ denote the space of all functions $f\in L^\infty(\mathbb R_+)$ such that the map $x\longmapsto f(\gamma/x)$ is 1-periodic. Then the weak-star closure in $L^\infty(\mathbb R_+)$ of the linear span of the functions $\{e^{\gamma}_n(x):=e^{2\pi in\gamma/x};~n\in\mathbb Z\}$ equals to $L^\infty_{\gamma}(\mathbb R_+).$ 
Observe that the functions in $L^\infty_{q}(\mathbb R_+)$ are defined freely on $[0,q],$ and 
because of periodicity they are uniquely determined on $\mathbb R_+\setminus[0,q].$ Similarly, 
the functions in $L^\infty_{\gamma}(\mathbb R_+)$ are defined freely on $\mathbb R_+\setminus[0,\gamma],$ 
and due to periodicity they are uniquely determined on $[0,\gamma].$ 
\smallskip

The operator $\text{O}_q : L^\infty([0,q])\rightarrow L^\infty(\mathbb R_+\setminus[0,q])$ is defined by 
\begin{equation}
\text{O}_q[\varphi](x)=\varphi\left(q\left\{x/q\right\}_1\right)\chi_{\mathbb R_+\setminus[0,q]}(x),~~
\text{where}~\varphi\in L^\infty([0,q]).
\end{equation}  
The operator $\text{T}_{\gamma} : L^\infty(\mathbb R_+\setminus[0,\gamma])\rightarrow L^\infty([0,\gamma])$ is defined by 
\begin{equation}
\text{T}_{\gamma}[\psi](x)=\psi\left(\frac{\gamma}{\{\gamma/x\}_1}\right)\chi_{[0,\gamma]\setminus\{0\}}(x),~~ 
\text{where}~\psi\in L^\infty(\mathbb R_+\setminus[0,\gamma]).
\end{equation} 
In view of the above facts, we get that 
$L^\infty_q(\mathbb R_+)=\left\{\varphi+\text{O}_q[\varphi]:\varphi\in L^\infty([0,q])\right\}$ and $L^\infty_{\gamma}(\mathbb R_+)=\left\{\psi+\text{T}_{\gamma}[\psi]:\psi\in L^\infty(\mathbb R_+\setminus[0,\gamma])\right\}.$ 

\subsubsection{} \textit{The Perron-Frobenius operators.}
For $q\leq\gamma<\infty,$ the \textit{Koopman operator} $\mathcal C_\gamma : L^\infty([0,q])\rightarrow L^\infty([0,q])$ associated to $U_\gamma$ be the map $\mathcal C_\gamma[\varphi](x)=\varphi\circ U_\gamma(x),$ 
where $\varphi\in L^\infty([0,q]).$ The predual adjoint of $\mathcal C_\gamma$ is the \textit{Perron-Frobenius}  
operator $\mathcal P_\gamma : L^1([0,q])\rightarrow L^1([0,q])$ given by 
\begin{equation}
\mathcal P_\gamma [h](x)=\sum\limits_{v=1}^\infty\frac{q\gamma}{(qv+x)^2}h\left(\frac{q\gamma}{qv+x}\right).
\end{equation}
The operator $\mathcal P_\gamma$ is linear and a norm contraction on $L^1([0,q]).$ Thus the point spectrum $\sigma_{point}(\mathcal P_\gamma)$ of $\mathcal P_\gamma$ is contained in $\bar{\mathbb D}.$ 
For $q<\gamma<\infty,$ consider the following restriction operators:  
\[\begin{cases} 
      \text{R}_5~:~L^\infty(\mathbb R_+\setminus[0,q])\rightarrow 
L^\infty(\mathbb R_+\setminus[0,\gamma]) \\
\text{R}_6~:~L^\infty([0,\gamma])\rightarrow 
L^\infty([0,q])\\
\text{R}_7~:~L^\infty([0,\gamma])\rightarrow 
L^\infty([0,\gamma]\setminus[0,q])\\
\text{R}_8~:~L^\infty(\mathbb R_+\setminus[0,q])\rightarrow 
L^\infty([0,\gamma]\setminus[0,q]) 
   \end{cases}
\] 
The corresponding pre-dual adjoints are the maps $\text{R}_5^*,\text{R}_6^*,\text{R}_7^*$ and 
$\text{R}_8^*$ respectively. As $\gamma>q,$ a simple calculation shows that $\mathcal C_\gamma^2=\text{R}_6\text{T}_{\gamma}\text{R}_5\text{O}_p$ and hence $\mathcal P_\gamma^2=\text{O}_q^*\text{R}_5^*\text{T}_{\gamma}^*\text{R}_6^*.$ 

\begin{proposition}\label{prop35}
For $q<\gamma<\infty,$ suppose $f\in L^1(\mathbb R_+)$ such that $f=f_1+f_2+f_3,$ where 
$f_1\in L^1([0,q]),$ $f_2\in L^1([0,\gamma]\setminus[0,q]),$ and 
$f_3\in L^1(\mathbb R_+\setminus[0,\gamma]).$ Then $f\in\mathcal K_\gamma^\perp$ 
if and only if $(i)~(\text{I}-\mathcal P_\gamma^2)f_1=\text{O}_q^*(-\text{R}_8^*+\text{R}_5^*\text{T}_{\gamma}^*\text{R}_7^*)f_2,$ where $\text{I}$ is the identity operator on $L^1([0,q]),$ and $(ii)~f_3=-\text{T}_{\gamma}^*\text{R}_6^*f_1-\text{T}_{\gamma}^*\text{R}_7^*f_2.$ 
\end{proposition}

\begin{proof}
The proof of Proposition \ref{prop35} works along the same lines as in the proof of 
(\cite{MHR}, p. 52, Proposition 6.1), hence omitted.
\end{proof}

\begin{remark}
For $0<\gamma<q,$ in analogy with the results in \cite{HR2} it seems likely that, the 
Perron-Frobenius operator $\mathcal P_\gamma$ has no eigenfunction corresponding to the 
eigenvalue $1.$ For the critical case $\gamma=q,$ $\mathcal P_\gamma$ has one-dimensional 
eigenspace corresponding to the eigenvalue $1.$ This in turn implies that $\mathcal{AC}\left(\Gamma_+,\Lambda_\gamma^\theta\right)=\{0\}$ for $0<\gamma<q,$ and for $\gamma=q,$ $\mathcal{AC}\left(\Gamma_+,\Lambda_\gamma^\theta\right)$ 
is one-dimensional. This problem is open. 
\end{remark}

\subsection{Proof of Theorem \ref{th12}}
The proof of Theorem \ref{th12} directly follows from the proof of Theorem \ref{th19}.  

\begin{theorem}\label{th19}
For $q<\gamma<\infty,$ there exits a bounded linear operator $\mathcal B_+~:~\text{BV}([0,\gamma]\setminus[0,q])\rightarrow L^1(\mathbb R_+)$ such that the range of $\mathcal B_+$ is infinite-dimensional, and 
contained in $\mathcal K_\gamma^\perp.$
\end{theorem}

\begin{proof}
We actually prove more precise statement, namely that, there exits a bounded linear operator 
$\mathcal B_+~:~\text{BV}([0,\gamma]\setminus[0,q])\rightarrow L^1(\mathbb R_+)$ such that   
$\mathcal B_+f(x)=f(x)$ a.e. $x\in [0,\gamma]\setminus[0,q],$ and for all 
$f\in\text{BV}([0,\gamma]\setminus[0,q]).$ Moreover, $\mathcal B_+$ has infinite-dimensional 
range which is contained in $\mathcal K_\gamma^\perp.$ The proof of Theorem \ref{th19} works 
along a similar path as in Theorem \ref{th14}, hence omitted.
\end{proof}

\bigskip

\noindent{\bf{Acknowledgements.}} The author wishes to thank E. K. Narayanan, Rama Rawat, R. K. Srivastava, 
and Sundaram Thangavelu for several valuable suggestions during the preparation of this manuscript. The author gratefully acknowledges the support provided by NBHM post-doctoral fellowship from the Department of Atomic Energy (DAE), Government of India. The author was supported by the Department of Mathematics, IISc Bangalore, India.

\bigskip

%\small

\end{document}